\documentclass[reqno,10pt]{amsart}
\usepackage{amscd,amssymb,verbatim,array}
\usepackage{hyperref}
\usepackage{amsmath}
\usepackage[active]{srcltx} 

\numberwithin{equation}{section} \theoremstyle{plain}

\newtheorem{theorem}{Theorem}[section]
\newtheorem{lemma}[theorem]{Lemma}

\newtheorem{corollary}[theorem]{Corollary}
\newtheorem{definition}[theorem]{Definition}

\theoremstyle{definition}

\theoremstyle{remark}

\numberwithin{equation}{section}

\newcommand{\E}{\mathcal{E}}

\newcommand{\even}{\operatorname{even}}

\newcommand{\cyl}{\operatorname{cyl}}
\newcommand{\odd}{\operatorname{odd}}

\newcommand{\Dim}{\operatorname{dim}}
\newcommand{\Mod}{\operatorname{mod}}

\newcommand{\rel}{\operatorname{rel}}

\newcommand{\Ker}{\operatorname{ker}}

\newcommand{\Tr}{\operatorname{Tr}}

\newcommand{\Imm}{\operatorname{Im}}
\newcommand{\Dom}{\operatorname{Dom}}

\newcommand{\Sign}{\operatorname{sign}}

\newcommand{\ddet}{\operatorname{det}}

\begin{document}

\title[Lefschetz fixed point formula for some boundary conditions]
{Lefschetz fixed point formula on a compact Riemannian manifold with boundary for some boundary conditions}
\author{Rung-Tzung Huang}

\address{Department of Mathematics, National Central University, Chung-Li 320, Taiwan, Republic of China}

\email{rthuang@math.ncu.edu.tw}

\author{Yoonweon Lee}

\address{Department of Mathematics, Inha University, Incheon, 402-751, Korea}

\email{yoonweon@inha.ac.kr}

\thanks{The first author was supported by the National Science Council, Republic of China with the grant number MOST 102-2115-M-008-005 and
the second author was supported by the National Research Foundation of Korea with the Grant number NRF-2012R1A1A2001086.}

\begin{abstract}
In [8] the authors introduced a pair of new de Rham complexes on a compact oriented Riemannian manifold with boundary
by using a pair of new boundary conditions
to discuss the refined analytic torsion on a compact manifold with boundary.
In this paper we discuss the Lefschetz fixed point formula on these complexes with respect to a smooth map
having simple fixed points and satisfying some special condition near the boundary.
For this purpose we are going to use the heat kernel method for the Lefschetz fixed point formula.
\end{abstract}
\maketitle

\section{Introduction}

\vspace{0.2 cm}

Let $(M, Y, g^{M})$ be an $m$-dimensional compact oriented Riemannian manifold with boundary $Y$ and $f : M \rightarrow M$ be a smooth map
such that $f(Y) \subset Y$.
A point $p \in M$ is said to be a simple fixed point of $f$ if

\begin{eqnarray}    \label{E:1.1}
f(p) = p, \qquad \ddet \left( I - df(p) \right) \neq 0.
\end{eqnarray}

\noindent
If $p$ is a simple fixed point, the graph of $f$ is transverse to the diagonal of $M \times M$ at $(p, p)$,
which implies that simple fixed points are discrete.
All through this paper we assume that all fixed points of $f$ are simple and hence $f$ has only finitely many fixed points.
For fixed points on the boundary $Y$, we need one more structure.
Let $f(x_{0}) = x_{0}$ with $x_{0} \in Y$. Then
$df(x_{0}) : T_{x_{0}}M \rightarrow T_{x_{0}}M$ induces a map $df_{Y}(x_{0}) : T_{x_{0}}Y \rightarrow T_{x_{0}}Y$.
We consider

\begin{eqnarray*}
a_{x_{0}} & = & df(x_{0}) (\Mod T_{x_{0}}Y) : T_{x_{0}} M / T_{x_{0}} Y  \rightarrow T_{x_{0}} M / T_{x_{0}} Y.
\end{eqnarray*}

\noindent
Since the quotient space $T_{x_{0}} M / T_{x_{0}} Y$ is one-dimensional, the map $a_{x_{0}}$ is simply multiplication by a number,
which we denote by $a_{x_{0}}$ again.
It's not difficult to see that $a_{x_{0}} \geq 0$ by considering the quotient space $T_{x_{0}} M / T_{x_{0}} Y$
as a normal half-line pointing inward at the boundary point $x_{0}$.
Moreover, since the fixed point $x_{0}$ is simple, $a_{x_{0}} \neq 1$ (see [5] for details).

\begin{definition}
(1) A simple boundary fixed point $x_{0} \in Y$ is called {\it attracting} if $a_{x_{0}} < 1$ and {\it repelling} if $a_{x_{0}} > 1$.  \newline
(2) We denote by ${\mathcal F}_{0}(f)$, ${\mathcal F}^{+}_{Y}(f)$ and ${\mathcal F}^{-}_{Y}(f)$
the set of all simple fixed points in the interior of $M$, the attracting fixed points in $Y$
and the repelling fixed points in $Y$, respectively. We denote ${\mathcal F}_{Y}(f) := {\mathcal F}^{+}_{Y}(f) \cup {\mathcal F}^{-}_{Y}(f)$
and ${\mathcal F}(f) := {\mathcal F}_{0}(f) \cup {\mathcal F}_{Y}(f)$.
\end{definition}

\noindent
A. V. Brenner and M. A. Shubin proved the following result in [5].

\begin{eqnarray}  \label{E:1.2}
\sum_{q=0}^{m} (-1)^{q} \Tr \left( f^{\ast} : H^{q}(M) \rightarrow H^{q}(M) \right) & = &
\sum_{p \in {\mathcal F}_{0}(f) \cup {\mathcal F}^{+}_{Y}(f)} \Sign \ddet \left( I - df(p) \right),  \nonumber  \\
\sum_{q=0}^{m} (-1)^{q} \Tr \left( f^{\ast} : H^{q}(M, Y) \rightarrow H^{q}(M, Y) \right) & = &
\sum_{p \in {\mathcal F}_{0}(f) \cup {\mathcal F}^{-}_{Y}(f)} \Sign \ddet \left( I - df(p) \right).
\end{eqnarray}

\noindent
This result extends the Atiyah-Bott-Lefschetz fixed point formula proven on a closed manifold in [1].

On the other hand, the authors introduced new de Rham complexes $\left( \Omega^{\bullet}_{{\widetilde {\mathcal P}}_{0}}(M), d \right)$ and
$\left( \Omega^{\bullet}_{{\widetilde {\mathcal P}}_{1}}(M), d \right)$ by using some boundary
conditions ${\widetilde {\mathcal P}}_{0}$ and ${\widetilde {\mathcal P}}_{1}$, which compute
$H^{q} \left( \Omega^{\bullet}_{{\widetilde {\mathcal P}}_{0}}(M), d \right) =
\begin{cases} H^{q}(M, Y) & \text{if} \quad q = \even  \\ H^{q}(M) & \text{if} \quad q = \odd \end{cases}$ and
$H^{q} \left( \Omega^{\bullet}_{{\widetilde {\mathcal P}}_{1}}(M), d \right) =
\begin{cases} H^{q}(M) & \text{if} \quad q = \even  \\ H^{q}(M, Y) & \text{if} \quad q = \odd \end{cases}$.
In this paper, we are going to discuss the Lefschetz fixed point formula on these complexes.
More precisely, when $f : M \rightarrow M$ is a smooth map having simple fixed points and satisfying some special condition near the boundary $Y$
(see Definition \ref{Definition:3.1}), we are going to describe

\begin{eqnarray*}
& & \sum_{q=\even} \Tr \left( f^{\ast} : H^{q}(M, Y) \rightarrow H^{q}(M, Y) \right) \hspace{0.1 cm} - \hspace{0.1 cm}
\sum_{q=\odd} \Tr \left( f^{\ast} : H^{q}(M) \rightarrow H^{q}(M) \right) \qquad  \text{and}  \\
& & \sum_{q=\even} \Tr \left( f^{\ast} : H^{q}(M) \rightarrow H^{q}(M) \right)  \hspace{0.1 cm} - \hspace{0.1 cm}
\sum_{q=\odd} \Tr \left( f^{\ast} : H^{q}(M, Y) \rightarrow H^{q}(M, Y) \right)
\end{eqnarray*}

\noindent
in terms of fixed points of $f$ and some additional data (see Theorem \ref{Theorem:3.3} below).
For this purpose, we are going to use the heat kernel method for the Lefschetz fixed point formula (cf. [3], [6]).

\vspace{0.2 cm}


\section{de Rham complex $( \Omega^{\bullet}_{{\widetilde {\mathcal P}}_{0}/{\widetilde {\mathcal P}}_{1}}(M), d )$ on a compact
Riemannian manifold with boundary}

\vspace{0.2 cm}

In this section we are going to introduce the
de Rham complex $( \Omega^{\bullet}_{{\widetilde {\mathcal P}}_{0}/{\widetilde {\mathcal P}}_{1}}(M), d )$ on a compact Riemannian manifold with boundary by using the boundary condition
${\widetilde {\mathcal P}}_{0}/{\widetilde {\mathcal P}}_{1}$.
We recall that $(M, Y, g^{M})$ is an $m$-dimensional compact oriented Riemannian manifold with boundary $Y$.
From now on, we assume that $g^{M}$ is a product metric near the boundary $Y$.
We denote by $d^{Y}_{q} : \Omega^{q}(Y) \rightarrow \Omega^{q+1}(Y)$ the de Rham operator induced from $d : \Omega^{q}(M) \rightarrow \Omega^{q+1}(M)$
and denote by $\star_{Y} : \Omega^{q}(Y) \rightarrow \Omega^{m-1-q}(Y)$ the Hodge star operator on $Y$ induced from the Hodge star operator $\star_{M}$ on $M$.
Then the formal adjoint $(d^{Y}_{q})^{\ast}$ of $d^{Y}_{q}$ is defined in the usual way.
We denote $\Delta_{Y}^{q} := (d^{Y}_{q})^{\ast} d^{Y}_{q} + d^{Y}_{q-1} (d^{Y}_{q-1})^{\ast}$ and ${\mathcal H}^{q}(Y) := \Ker \Delta_{Y}^{q}$.
By the Hodge decomposition, we have

\begin{eqnarray*}
\Omega^{q}(Y) & = & \Imm d^{Y}_{q-1} \oplus {\mathcal H}^{q}(Y) \oplus \Imm (d^{Y}_{q})^{\ast}
\end{eqnarray*}

\noindent
Let $N$ be a collar neighborhood of $Y$ which is isometric to $[0,1) \times Y$
and $u$ be the coordinate normal to the boundary $Y$ on $N$.
If $d \phi =  d^\ast \phi = 0$ for $\phi \in \Omega^{q}(M)$, simple computation shows that $\phi$ is expressed
on the boundary $Y$ by

\begin{equation} \label{E:2.1}
\phi|_{Y} = \left( d^{Y} \varphi_{1} + \varphi_{2} \right) + du \wedge \left(  d^{Y \ast}  \psi_{1} + \psi_{2} \right), \quad
\varphi_{1}, \hspace{0.1 cm} \psi_{1} \in \Omega^{\bullet}(Y),
\quad \varphi_{2}, \hspace{0.1 cm} \psi_{2} \in {\mathcal H}^{\bullet}(Y).
\end{equation}

\noindent
In other words, $\varphi_{2}$ and $\psi_{2}$ are harmonic parts of $\iota^{\ast} \phi$ and $\star_{Y} \iota^{\ast} ( \star_{M} \phi )$ up to sign,
where $ \iota : Y \rightarrow M$ is the natural inclusion.
We denote ${\mathcal K}^{q}$ and ${\mathcal K}$ by

\begin{equation} \label{E:2.2}
{\mathcal K}^{q} := \{ \varphi_{2} \in {\mathcal H}^{q}(Y) \mid d \phi =  d^\ast \phi = 0 \},  \qquad
{\mathcal K}  := \oplus_{q=0}^{m-1} {\mathcal K}^{q},
\end{equation}

\noindent
where $\phi$ has the form (\ref{E:2.1}).
If $d \phi =  d^\ast \phi = 0$ for $\phi \in \Omega^{q}(M)$,
 $d (\star_{M} \phi) =  d^\ast (\star_{M} \phi) = 0$, which implies that

\begin{equation} \label{E:2.3}
 \star_{Y} {\mathcal K}^{m-q} = \{ \psi_{2} \in {\mathcal H}^{q-1}(Y) \mid d \phi = d^\ast \phi = 0 \},
\end{equation}

\noindent
where $\phi$ has the form (\ref{E:2.1}).
We have the following lemma, whose proof we refer to Lemma 2.4 in [8].

\vspace{0.2 cm}

\begin{lemma} \label{Lemma:2.1}
 ${\mathcal K}$ is orthogonal to $\star_{Y} \mathcal{K}$ and
${\mathcal K} \oplus (\star_{Y} {\mathcal K}) = {\mathcal H}^{\bullet}(Y)$.
\end{lemma}

\vspace{0.2 cm}

\noindent
We consider the homomorphism $\iota^{\ast} : H^{\bullet}(M) \rightarrow H^{\bullet}(Y)$ induced from the natural inclusion $\iota : Y \rightarrow M$.
It is well known that each cohomology class $[\omega] \in H^{\bullet}(M)$ has a unique representative $\omega_{0} \in \Omega^{\bullet}(M)$ such that
$d \omega_{0} = d^{\ast} \omega_{0} = 0$ and $\iota^{\ast} (\star_{M} \omega_{0}) = 0$ (see Theorem 2.7.3 in [6]).
Since $\iota^{\ast} \omega_{0}$ is a closed form, $[\iota^{\ast} \omega_{0}] \in H^{\bullet}(Y)$.
We denote by $(\iota^{\ast} \omega_{0})_{h}$ the harmonic part of $\iota^{\ast} \omega_{0}$
and define a map

\begin{eqnarray*}
{\mathcal G } : \Imm \left( \iota^{\ast} : H^{\bullet}(M) \rightarrow H^{\bullet}(Y) \right) \rightarrow {\mathcal K}, \qquad
{\mathcal G } ([\iota^{\ast} \omega_{0}]) = (\iota^{\ast} \omega_{0})_{h} .
\end{eqnarray*}

\noindent
A standard argument using the Lefschetz-Poincar\'e duality shows that
$\Dim \Imm \left( \iota^{\ast} : H^{\bullet}(M) \rightarrow H^{\bullet}(Y) \right)$ is equal to $\frac{1}{2} \Dim H^{\bullet}(Y)$.
Since ${\mathcal G}$ is a monomorphism, this fact together with Lemma \ref{Lemma:2.1} shows that
${\mathcal G}$ is an isomorphism. Summarizing this fact, we have the following result (cf. Corollary 8.4 in [9]).

\begin{lemma} \label{Lemma:2.2}
For each $q$, ${\mathcal K}^{q}$ can be naturally identified with $\Imm \left( \iota^{\ast} : H^{q}(M) \rightarrow H^{q}(Y) \right)$.
\end{lemma}

\vspace{0.2 cm}

We next consider the natural isomorphism

\begin{equation}  \label{E:2.4}
\Psi : \Omega^{p}(N) \rightarrow C^{\infty}([0, 1), \Omega^{p}(Y) \oplus \Omega^{p-1}(Y)),  \qquad
\Psi(\omega_{1} + du \wedge \omega_{2}) = \left( \begin{array}{clcr} \omega_{1}  \\ \omega_{2} \end{array} \right).
\end{equation}

\noindent
We put ${\mathcal L}_{0} := \left( \begin{array}{clcr} {\mathcal K}  \\ {\mathcal K} \end{array} \right)$,
${\mathcal L}_{1} := \left( \begin{array}{clcr} {\star_{Y} \mathcal K}  \\ \star_{Y} {\mathcal K} \end{array} \right)$
 and consider the orthogonal projections defined by

\begin{eqnarray*}
& & \hspace{1.0 cm} {\mathcal P}_{-, {\mathcal L}_{0}}, \hspace{0.1 cm} {\mathcal P}_{+, {\mathcal L}_{1}} :
 \Omega^{\bullet}(Y) \oplus \Omega^{\bullet}(Y)  \rightarrow \Omega^{\bullet}(Y) \oplus \Omega^{\bullet}(Y)  \\
& & \Imm {\mathcal P}_{-, {\mathcal L}_{0}} = \left( \begin{array}{clcr} \Imm d^{Y} \oplus {\mathcal K} \\ \Imm d^{Y} \oplus {\mathcal K} \end{array} \right),
\qquad  \Imm {\mathcal P}_{+, {\mathcal L}_{1}} =
\left( \begin{array}{clcr} \Imm (d^{Y})^{\ast} \oplus \star_{Y} {\mathcal K} \\ \Imm (d^{Y})^{\ast} \oplus \star_{Y} {\mathcal K} \end{array} \right) .
\end{eqnarray*}

\noindent
We then define the spaces of differential forms satisfying the boundary conditions
${\mathcal P}_{-, {\mathcal L}_{0}}$ and ${\mathcal P}_{+, {\mathcal L}_{1}}$
by

\begin{eqnarray*}
\Omega^{q}_{{\mathcal P}_{-, {\mathcal L}_{0}}}(M) & := & \{ \phi \in \Omega^{q}(M) \mid {\mathcal P}_{-, {\mathcal L}_{0}} ( \phi|_{Y} ) = 0, \quad
{\mathcal P}_{-, {\mathcal L}_{0}} (( \star_{M} ( d + d^{\ast} ) \phi)|_{Y} ) = 0 \},  \\
\Omega^{q}_{{\mathcal P}_{+, {\mathcal L}_{1}}}(M) & := & \{ \phi \in \Omega^{q}(M) \mid {\mathcal P}_{+, {\mathcal L}_{1}} ( \phi|_{Y} ) = 0, \quad
{\mathcal P}_{+, {\mathcal L}_{1}} (( \star_{M} ( d + d^{\ast} ) \phi)|_{Y} ) = 0 \},
\end{eqnarray*}

\noindent
and also define

\begin{eqnarray}  \label{E:2.5}
\Omega^{q, \infty}_{{\mathcal P}_{-, {\mathcal L}_{0}}}(M) & = & \{ \phi \in \Omega^{q}(M) \mid
{\mathcal P}_{-, {\mathcal L}_{0}} \left( \left( ( \star_{M} (d + d^{\ast}))^{l} \phi \right)|_{Y}\right) = 0, \quad l = 0, 1, 2, \cdots  \}, \nonumber \\
\Omega^{q, \infty}_{{\mathcal P}_{+, {\mathcal L}_{1}}}(M) & = & \{ \phi \in \Omega^{q}(M) \mid
{\mathcal P}_{+, {\mathcal L}_{1}} \left( \left( ( \star_{M} (d + d^{\ast}))^{l} \phi \right)|_{Y}\right) = 0, \quad l = 0, 1, 2, \cdots  \}.
\end{eqnarray}

\vspace{0.2 cm}

Simple computation shows that if $\phi \in \Omega^{q}_{{\mathcal P}_{-, {\mathcal L}_{0}}}(M)$, then $\star_{M} \phi \in \Omega^{m-q}_{{\mathcal P}_{+, {\mathcal L}_{1}}}(M)$ and vice versa.
Similarly, for each $\phi \in \Omega^{q}_{{\mathcal P}_{-, {\mathcal L}_{0}}}(M)$ and $\psi \in \Omega^{q}_{{\mathcal P}_{+, {\mathcal L}_{1}}}(M)$,
we have

\begin{eqnarray}   \label{E:2.6}
{\mathcal P}_{+, {\mathcal L}_{1}} ( (d \phi )|_{Y} ) = 0 \qquad \text{and} \qquad {\mathcal P}_{-, {\mathcal L}_{0}} ( (d \psi )|_{Y} ) = 0.
\end{eqnarray}

\noindent
These imply that $\star_{M}$ maps $\hspace{0.1 cm} \Omega^{q, \infty}_{{\mathcal P}_{-, {\mathcal L}_{0}}}(M)$
($\Omega^{q, \infty}_{{\mathcal P}_{+, {\mathcal L}_{1}}}(M)$) into
$\hspace{0.1 cm} \Omega^{m-q, \infty}_{{\mathcal P}_{+, {\mathcal L}_{1}}}(M)$ ($\Omega^{m-q, \infty}_{{\mathcal P}_{-, {\mathcal L}_{0}}}(M)$) and
$d$ maps $\hspace{0.1 cm} \Omega^{q, \infty}_{{\mathcal P}_{-, {\mathcal L}_{0}}}(M)$
($\Omega^{q, \infty}_{{\mathcal P}_{+, {\mathcal L}_{1}}}(M)$) into
$\hspace{0.1 cm} \Omega^{q+1, \infty}_{{\mathcal P}_{+, {\mathcal L}_{1}}}(M)$ ($\Omega^{q+1, \infty}_{{\mathcal P}_{-, {\mathcal L}_{0}}}(M)$).

\begin{definition} \label{Definition:2.1}
We define projections ${\widetilde {\mathcal P}}_{0}$,
${\widetilde {\mathcal P}}_{1} : \Omega^{\bullet}(Y) \oplus \Omega^{\bullet}(Y) \rightarrow
\Omega^{\bullet}(Y) \oplus \Omega^{\bullet}(Y)$ as follows.
For $\phi \in \Omega^{q}(M, E)$
$$
{\widetilde {\mathcal P}}_{0} (\phi|_{Y}) = \begin{cases} {\mathcal P}_{-, {\mathcal L}_{0}} (\phi|_{Y}) \quad
\text{if} \quad q \quad \text{is} \quad \text{even} \\
{\mathcal P}_{+, {\mathcal L}_{1}} (\phi|_{Y}) \quad \text{if} \quad q \quad \text{is} \quad \text{odd} ,
\end{cases}
\qquad
{\widetilde {\mathcal P}}_{1} (\phi|_{Y}) = \begin{cases} {\mathcal P}_{+, {\mathcal L}_{1}} (\phi|_{Y}) \quad
\text{if} \quad q \quad \text{is} \quad \text{even} \\
{\mathcal P}_{-, {\mathcal L}_{0}} (\phi|_{Y}) \quad \text{if} \quad q \quad \text{is} \quad \text{odd} .
\end{cases}
$$

\end{definition}

\noindent

\noindent
Then the above argument leads to the following cochain complexes

\begin{eqnarray}
(\Omega^{\bullet, \infty}_{{\widetilde {\mathcal P}}_{0}}(M), \hspace{0.1 cm} d) & : &
 0 \longrightarrow \Omega^{0, \infty}_{{\mathcal P}_{-, {\mathcal L}_{0}}}(M) \stackrel{d}{\longrightarrow}
\Omega^{1, \infty}_{{\mathcal P}_{+, {\mathcal L}_{1}}}(M) \stackrel{d}{\longrightarrow}
\Omega^{2, \infty}_{{\mathcal P}_{-, {\mathcal L}_{0}}}(M) \stackrel{d}{\longrightarrow}
\cdots \longrightarrow 0.   \label{E:2.7} \\
(\Omega^{\bullet, \infty}_{{\widetilde {\mathcal P}}_{1}}(M), \hspace{0.1 cm} d) & : &
 0 \longrightarrow \Omega^{0, \infty}_{{\mathcal P}_{+, {\mathcal L}_{1}}}(M) \stackrel{d}{\longrightarrow}
\Omega^{1, \infty}_{{\mathcal P}_{-, {\mathcal L}_{0}}}(M) \stackrel{d}{\longrightarrow}
\Omega^{2, \infty}_{{\mathcal P}_{+, {\mathcal L}_{1}}}(M) \stackrel{d}{\longrightarrow}
\cdots  \longrightarrow 0.    \label{E:2.8}
\end{eqnarray}

\vspace{0.2 cm}

\noindent
\noindent
We define the Laplacians $\Delta^{q}_{{\widetilde {\mathcal P}}_{0}}$ and $\Delta^{q}_{{\widetilde {\mathcal P}}_{1}}$ by

\begin{eqnarray*}
\Delta^{q} := d_{q}^{\ast} d_{q} + d_{q-1} d_{q-1}^{\ast}, \qquad
\Dom \left( \Delta^{q}_{{\widetilde {\mathcal P}}_{0}}  \right) \hspace{0.1 cm} = \hspace{0.1 cm} \Omega^{q, \infty}_{{\widetilde {\mathcal P}}_{0}}(M)
\hspace{0.1 cm} = \hspace{0.1 cm}
\begin{cases} \Omega^{q, \infty}_{{\mathcal P}_{-, {\mathcal L}_{0}}}(M) & \text{for} \hspace{0.2 cm} q \hspace{0.2 cm} \even \\
\Omega^{q, \infty}_{{\mathcal P}_{+, {\mathcal L}_{1}}}(M) & \text{for} \hspace{0.2 cm} q \hspace{0.2 cm} \odd . \end{cases}
\end{eqnarray*}

\noindent
We define $\Dom \left( \Delta^{q}_{{\widetilde {\mathcal P}}_{1}}  \right)$ in the same way.
It is not difficult to see that ${\mathcal P}_{-, {\mathcal L}_{0}}$ and ${\mathcal P}_{+, {\mathcal L}_{1}}$ are well-posed boundary conditions
for the odd signature opertator and Laplacian in the sense of Seeley  ([7], [11]). We refer to Lemma 2.15 in [8] for details.
Hence, $\Delta^{q}_{{\widetilde {\mathcal P}}_{0}}$ and $\Delta^{q}_{{\widetilde {\mathcal P}}_{1}}$  have compact resolvents and discrete spectra.
Moreover, the Green formula shows that $\Delta^{q}_{{\widetilde {\mathcal P}}_{0}}$ and $\Delta^{q}_{{\widetilde {\mathcal P}}_{1}}$
are formally self-adjoint and non-negative.
The following lemma is straightforward (see Lemma 2.11 in [8] for details).

\vspace{0.2 cm}

\begin{lemma}    \label{Lemma:2.3}
The cohomologies of the complex $(\Omega^{\bullet, \infty}_{{\widetilde {\mathcal P}}_{0}/{\widetilde {\mathcal P}}_{1}}(M), \hspace{0.1 cm} d)$
are given as follows.
\begin{eqnarray}     \label{E:2.9}
H^{q}((\Omega^{\bullet, \infty}_{{\widetilde {\mathcal P}}_{0}}(M), \hspace{0.1 cm} d)) & = &
\Ker \Delta^{q}_{{\widetilde {\mathcal P}}_{0}} =
\begin{cases} H^{q}(M, Y) \quad \text{if} \quad q \quad \text{is} \quad \text{even} \\
H^{q}(M) \quad \text{if} \quad q \quad \text{is} \quad \text{odd} ,  \end{cases}    \nonumber \\
H^{q}((\Omega^{\bullet, \infty}_{{\widetilde {\mathcal P}}_{1}}(M), \hspace{0.1 cm} d)) & = &
\Ker \Delta^{q}_{{\widetilde {\mathcal P}}_{1}} =
\begin{cases} H^{q}(M) \quad \text{if} \quad q \quad \text{is} \quad \text{even} \\
H^{q}(M, Y) \quad \text{if} \quad q \quad \text{is} \quad \text{odd} .   \end{cases}
\end{eqnarray}
\end{lemma}

\begin{proof}
We denote by ${\mathcal H}_{\rel}^{q}(M) :=
\{ \phi = \phi_{1} + du \wedge \phi_{2} \in \Omega^{q}(M) \mid d \phi = d^{\ast} \phi = 0, \phi_{1}|_{Y} = 0 \}$
the space of harmonic $q$-forms satisfying the relative boundary condition. It is well known that ${\mathcal H}_{\rel}^{q}(M)$
is isomorphic to the singular cohomology $H^{q}(M, Y)$.
The Green theorem shows that
$\Ker \Delta^{q}_{{\mathcal P}_{-, {\mathcal L}_{0}}} = \{ \phi \in \Omega^{q}(M) \mid d \phi = d^{\ast} \phi = 0,
{\mathcal P}_{-, {\mathcal L}_{0}} (\phi|_{Y}) = 0 \}$.
We are going to show that $\Ker \Delta^{q}_{{\mathcal P}_{-, {\mathcal L}_{0}}} = {\mathcal H}_{\rel}^{q}(M)$.
Let $\phi = \phi_{1} + du \wedge \phi_{2} \in {\mathcal H}_{\rel}^{q}(M)$.
Then by (\ref{E:2.1}) with the fact that $\phi_{1}|_{Y} = 0$,
we have $\phi|_{Y} = du \wedge \left( d^{Y \ast} \psi_{1} + \psi_{2} \right)$, which shows that
${\mathcal P}_{-, {\mathcal L}_{0}} (\phi|_{Y}) = 0$. Hence, $\phi \in \Ker \Delta^{q}_{{\mathcal P}_{-, {\mathcal L}_{0}}}$.
Conversely, let $\phi = \phi_{1} + du \wedge \phi_{2} \in \Ker \Delta^{q}_{{\mathcal P}_{-, {\mathcal L}_{0}}}$.
By (\ref{E:2.1}) with the fact that ${\mathcal P}_{-, {\mathcal L}_{0}} (\phi|_{Y}) = 0$,
we have $\phi|_{Y} = du \wedge \left( d^{Y \ast} \psi_{1} + \psi_{2} \right)$,
which shows that $\phi \in {\mathcal H}_{\rel}^{q}(M)$.
Other cases can be checked in the same way. This completes the proof of the lemma.
\end{proof}

\noindent
In the next section, we discuss the Lefschetz fixed point formula on the complexes (\ref{E:2.7}) and (\ref{E:2.8}).

\vspace{0.2 cm}

\section{Lefschetz fixed point formula on the complex
$(\Omega^{\bullet, \infty}_{{\widetilde {\mathcal P}}_{0}/{\widetilde {\mathcal P}}_{1}}(M), \hspace{0.1 cm} d)$}

\vspace{0.2 cm}

We recall that $g^{M}$ is assumed to be a product metric near $Y$ and
begin with the following definition.

\begin{definition}  \label{Definition:3.1}
For a smooth map $f : M \rightarrow M$, $f$ is said to satisfy the Condition A
if on some collar neighborhood $[0, \epsilon) \times Y$ of $Y$, $f : [0, \epsilon) \times Y \rightarrow M$ is expressed by
$f (u, y) = ( c u, B(y) )$, where $c$ is a positive real number which is not equal to $1$ and
$B : (Y, g^{Y}) \rightarrow (Y, g^{Y})$ is an isometry.
\end{definition}

\vspace{0.2 cm}

\noindent
{\it Remark} : If $f : M \rightarrow M$ satisfies the Condition A, then all the fixed points in $Y$ are attracting
if $0 < c < 1$ and repelling if $c > 1$.

\vspace{0.2 cm}

\noindent
If $f$ satisfies the Condition A, for $\omega = \omega_{1} + du \wedge \omega_{2}$ on a collar neighborhood of $Y$,
$f^{\ast} \omega = B^{\ast} \omega_{1} + c du \wedge B^{\ast} \omega_{2}$.
Since $B$ is an isometry, $B^{\ast}$ maps $\Imm d^{Y}$ and $\Imm (d^{Y})^{\ast}$ onto $\Imm d^{Y}$ and $\Imm (d^{Y})^{\ast}$, respectively.
The following lemma shows that
$f^{\ast}$ maps $\Omega^{\bullet, \infty}_{{\widetilde {\mathcal P}}_{0}}(M)$ into $\Omega^{\bullet, \infty}_{{\widetilde {\mathcal P}}_{0}}(M)$
and maps $\Omega^{\bullet, \infty}_{{\widetilde {\mathcal P}}_{1}}(M)$ into $\Omega^{\bullet, \infty}_{{\widetilde {\mathcal P}}_{1}}(M)$.

\vspace{0.2 cm}

\begin{lemma}  \label{Lemma:3.1}
$B^{\ast}$ maps ${\mathcal K}^{q}$ onto ${\mathcal K}^{q}$ and $\star_{Y} {\mathcal K}^{q}$ onto $\star_{Y} {\mathcal K}^{q}$.
\end{lemma}

\begin{proof}
Since $B$ is an isometry,
it is enough to show that $B^{\ast}$ maps ${\mathcal K}^{q}$ into ${\mathcal K}^{q}$.
The following commutative diagrams show that for $[\omega] \in H^{q}(M)$, $B^{\ast} \iota^{\ast} \omega = \iota^{\ast} f^{\ast} \omega$.

\begin{eqnarray*}
\begin{CD}
  Y      & @> \iota >>  & M \\
  @V B VV   \circlearrowright     &       & @VV f V \\
 Y     & @> \iota >>  &  M
\end{CD}
\qquad  \Longrightarrow  \qquad
\begin{CD}
  H^{q}(M)      & @> \iota^{\ast} >>  & H^{q}(Y) \\
  @V f^{\ast} VV   \circlearrowright     &       & @VV B^{\ast} V \\
 H^{q}(M)     & @> \iota^{\ast} >>  &  H^{q}(Y)
\end{CD}
\end{eqnarray*}

\noindent
This fact together with Lemma \ref{Lemma:2.2} implies the result.
\end{proof}

\vspace{0.2 cm}

Since $f^{\ast}$ commutes with $d$, $f^{\ast} : (\Omega^{\bullet, \infty}_{{\widetilde {\mathcal P}}_{0}/{\widetilde {\mathcal P}}_{1}}(M),
\hspace{0.1 cm} d) \rightarrow (\Omega^{\bullet, \infty}_{{\widetilde {\mathcal P}}_{0}/{\widetilde {\mathcal P}}_{1}}(M), \hspace{0.1 cm} d)$
is a cochain map. In this section we are going to discuss the Lefschetz fixed point formula on these complexes for smooth maps
having only simple fixed points and satisfying the Condition A.

\begin{definition}  \label{Definition:3.2}
Suppose that $f : M \rightarrow M$ is a smooth map satisfying the Condition A.
We define the Lefschetz number of $f$ with respect to the complex
$(\Omega^{\bullet, \infty}_{{\widetilde {\mathcal P}}_{i}}(M), \hspace{0.1 cm} d)$ ($i = 0, 1$) by
\begin{eqnarray*}
L_{\widetilde{\mathcal P}_{i}}(f) & = &
\sum_{q=0}^m (-1)^{q} \Tr \left( f^* : H^{q}((\Omega^{\bullet, \infty}_{{\widetilde {\mathcal P}}_{i}}(M), \hspace{0.1 cm} d)) \rightarrow
H^{q}((\Omega^{\bullet, \infty}_{{\widetilde {\mathcal P}}_{i}}(M), \hspace{0.1 cm} d)) \right).
\end{eqnarray*}
\end{definition}

\vspace{0.2 cm}

We are going to express $L_{\widetilde{\mathcal P}_{i}}(f)$ in terms of fixed points of $f$ and some additional data.
We consider $L_{\widetilde{\mathcal P}_{0}}(f)$ first.
Using Lemma \ref{Lemma:2.3} and the standard argument for the trace of a heat operator (see Lemma 1.10.1 in [6] or Theorem 4 in [3] for details), we have

\begin{eqnarray}   \label{E:3.1}
L_{\widetilde{\mathcal P}_{0}}(f)  & = &
\sum_{q = \even} \Tr \left( f^{\ast} : H^{q}(M, Y) \rightarrow H^{q}(M, Y) \right) \hspace{0.1 cm} - \hspace{0.1 cm}
\sum_{q=\odd} \Tr \left( f^{\ast} : H^{q}(M) \rightarrow H^{q}(M) \right)    \\
& = & \sum_{q=0}^{m} (-1)^{q} \Tr \left( f^{\ast} e^{- t \Delta^{q}_{\widetilde{\mathcal P}_{0}}} \right)
\hspace{0.1 cm} = \hspace{0.1 cm} \lim_{t \rightarrow 0} \sum_{q=0}^{m} (-1)^{q} \Tr \left( f^{\ast} e^{- t \Delta^{q}_{\widetilde{\mathcal P}_{0}}} \right)
\nonumber   \\
& = & \lim_{t \rightarrow 0} \left\{ \sum_{q= \even} \Tr \left( f^{\ast} e^{- t \Delta^{q}_{{\mathcal P}_{-, {\mathcal L}_{0}}}} \right)
\hspace{0.1 cm} - \hspace{0.1 cm}  \sum_{q= \odd} \Tr \left( f^{\ast} e^{- t \Delta^{q}_{{\mathcal P}_{+, {\mathcal L}_{1}}}} \right)  \right\}  \nonumber  \\
& = & \lim_{t \rightarrow 0} \int_{M} \left\{ \sum_{q= \even}  \Tr
\left( {\mathcal T}_{q} (x) {\mathcal E}^{q}_{{\mathcal P}_{-, {\mathcal L}_{0}}} (t, f(x), x) \right)
\hspace{0.1 cm} - \hspace{0.1 cm}
\sum_{q= \odd} \Tr \left( {\mathcal T}_{q}(x) {\mathcal E}^{q}_{{\mathcal P}_{+, {\mathcal L}_{1}}} (t, f(x), x) \right) \right\} d vol(x),  \nonumber
\end{eqnarray}

\noindent
where ${\mathcal T}_{q}(x) := \Lambda^{q}( (df (x))^{T} ) : \Lambda^{q} T^{\ast}_{f(x)} M \rightarrow \Lambda^{q} T^{\ast}_{x} M$ is the pull-back map
mapping the fiber over $f(x)$ to the fiber over $x$ and
${\mathcal E}^{q}_{{\mathcal P}_{-, {\mathcal L}_{0}}/{\mathcal P}_{+, {\mathcal L}_{1}}} (t, x, z)$ is the kernel of
$e^{- t \Delta^{q}_{{\mathcal P}_{-, {\mathcal L}_{0}}/{\mathcal P}_{+, {\mathcal L}_{1}}}}$.
We choose $\epsilon > 0$ such that
$( [0, 2 \epsilon) \times Y ) \cap {\mathcal F}(f) = {\mathcal F}_{Y}(f)$.
For each $x \in {\mathcal F}_{0}(f)$, choose a small open neighborhood $U_{x}$ of $x$ such that $U_{x} \cap ( [0, \epsilon) \times Y ) = \emptyset$.
Putting $W:= M - \left( \cup_{x \in {\mathcal F}_{0}(f)} U_{x} \cup [0, \frac{\epsilon}{7}) \times Y \right)$,
the standard argument (see Lemma 1.10.2 in [6] or Theorem 5 in [3] for details)  shows that

\begin{eqnarray}   \label{E:3.2}
\lim_{t \rightarrow 0} \int_{W}
\Tr \left( {\mathcal T}_{q}(x) {\mathcal E}^{q}_{{\mathcal P}_{-, {\mathcal L}_{0}}/{\mathcal P}_{+, {\mathcal L}_{1}}} (t, f(x), x) \right) d vol(x)
& = & 0.
\end{eqnarray}

\noindent
Hence, we can rewrite (\ref{E:3.1}) as follows.

\begin{eqnarray}   \label{E:3.3}
 L_{\widetilde{\mathcal P}_{0}}(f)  & = &
 \lim_{t \rightarrow 0} \sum_{x \in {\mathcal F}_{0}(f)} \sum_{q= \even} \int_{U_{x}}
\Tr \left( {\mathcal T}_{q}(x) {\mathcal E}^{q}_{{\mathcal P}_{-, {\mathcal L}_{0}}} (t, f(x), x) \right) d vol(x)  \nonumber   \\
& - &  \lim_{t \rightarrow 0} \sum_{x \in {\mathcal F}_{0}(f)}
\sum_{q= \odd} \int_{U_{x}} \Tr \left( {\mathcal T}_{q}(x) {\mathcal E}^{q}_{{\mathcal P}_{+, {\mathcal L}_{1}}} (t, f(x), x) \right) d vol(x)   \nonumber  \\
& + &
\lim_{t \rightarrow 0}  \sum_{q= \even} \int_{Y} \int_{0}^{\frac{\epsilon}{7}} \Tr
\left( {\mathcal T}_{q}(x) {\mathcal E}^{q}_{{\mathcal P}_{-, {\mathcal L}_{0}}} (t, f(x), x) \right) du \hspace{0.1 cm} d vol(y)  \nonumber \\
& - &  \lim_{t \rightarrow 0}
\sum_{q= \odd} \int_{Y} \int_{0}^{\frac{\epsilon}{7}} \Tr \left( {\mathcal T}_{q}(x) {\mathcal E}^{q}_{{\mathcal P}_{+, {\mathcal L}_{1}}} (t, f(x), x) \right)
du \hspace{0.1 cm} d vol(y)  .
\end{eqnarray}

We next construct the parametrix $Q^{q}_{{\mathcal P}_{-, {\mathcal L}_{0}}/{\mathcal P}_{+, {\mathcal L}_{1}}}(t, x, z)$ of the heat kernel
${\mathcal E}^{q}_{{\mathcal P}_{-, {\mathcal L}_{0}}/{\mathcal P}_{+, {\mathcal L}_{1}}} (t, x, z)$
by combining the interior contribution and the boundary contribution.
We denote by ${\widetilde M}$ the closed double of $M$, {\it i.e.}, ${\widetilde M} = M \cup_{Y} M$ and
extend the Laplacian $\Delta^{q}$ on $M$ to the Laplacian on ${\widetilde M}$, denoted by ${\widetilde \Delta}^{q}$.
Let ${\widetilde \E_{q}}(t, x, z)$ be the kernel of the heat operator $e^{-t {\widetilde \Delta}^{q}}$.
It is well known (for example, p.225 in [4]) that

\begin{eqnarray} \label{E:3.4}
| {\widetilde \E_{q}}(t, x, z) | \hspace{0.1 cm} \leq  \hspace{0.1 cm}   c_{1} t^{- \frac{m}{2}} e^{- c_{2} \frac{d(x, z)^{2}}{t}} ,
\end{eqnarray}

\noindent
where $c_{i}$'s are some positive constants.

Let $N_{\infty} := [0, \infty) \times Y$ be a half infinite cylinder and
$\Delta^{q}_{N_{\infty}} := - \partial_{u}^{2} + \left( \begin{array}{clcr} \Delta_{Y}^{q} \\ \Delta^{q-1}_{Y} \end{array} \right) $
be the Laplacian acting on $q$-forms on $N_{\infty}$.
We decompose $\Omega^{q}(Y)$ by  $\Omega^{q}(Y)  =  \Omega^{q}_{-}(Y) \oplus \Omega^{q}_{+}(Y)$, where

\begin{eqnarray}   \label{E:3.5}
\Omega^{q}_{-}(Y) & := & \left( \Imm d^{Y} \oplus {\mathcal K} \right) \cap \Omega^{q}(Y), \qquad
\Omega^{q}_{+}(Y) \hspace{0.1 cm} := \hspace{0.1 cm} \left( \Imm (d^{Y})^{\ast} \oplus \star_{Y} {\mathcal K} \right) \cap \Omega^{q}(Y).
\end{eqnarray}

\noindent
We denote by $\{ \phi_{q, j} \mid j = 1, 2, \cdots \}$ and $\{ \psi_{q, j} \mid j = 1, 2, \cdots \}$ the orthonormal bases of
$\Omega^{q}_{-}(Y)$ and $\Omega^{q}_{+}(Y)$ consisting of eigenforms of $\Delta_{Y}^{q}$ with eigenvalues
$\{ \lambda_{q, j} \mid j = 1, 2, \cdots \}$ and $\{ \mu_{q, j} \mid j = 1, 2, \cdots \}$, respectively.
Then the heat kernels ${\mathcal E}^{\cyl, q}_{{\mathcal P}_{-, {\mathcal L}_{0}}}$ and ${\mathcal E}^{\cyl, q}_{{\mathcal P}_{+, {\mathcal L}_{1}}}$ of $\Delta^{q}_{N_{\infty}}$ with respect to the boundary conditions
${\mathcal P}_{-, {\mathcal L}_{0}}$ and ${\mathcal P}_{+, {\mathcal L}_{1}}$ on $\{ 0 \} \times Y$ are given as follows (cf. p.226 in [4]).

\begin{eqnarray}   \label{E:3.6}
{\mathcal E}^{\cyl, q}_{{\mathcal P}_{-, {\mathcal L}_{0}}} (t, (u, y), (v, y^{\prime}))
& = & \sum_{j=1}^{\infty} \frac{e^{-t \lambda_{q, j}}}{\sqrt{4 \pi t}} \left( e^{- \frac{(u - v)^{2}}{4t}} - e^{- \frac{(u + v)^{2}}{4t}} \right)
\phi_{q, j}(y) \otimes \phi^{\ast}_{q, j}(y^{\prime})     \\
& + & \sum_{j=1}^{\infty} \frac{e^{-t \mu_{q, j}}}{\sqrt{4 \pi t}} \left( e^{- \frac{(u - v)^{2}}{4t}} + e^{- \frac{(u + v)^{2}}{4t}} \right)
\psi_{q, j}(y) \otimes \psi^{\ast}_{q, j}(y^{\prime})   \nonumber   \\
& + & \sum_{j=1}^{\infty} \frac{e^{-t \lambda_{q-1, j}}}{\sqrt{4 \pi t}} \left( e^{- \frac{(u - v)^{2}}{4t}} - e^{- \frac{(u + v)^{2}}{4t}} \right)
(du \wedge \phi_{q-1, j}(y)) \otimes (dv \wedge \phi_{q-1, j}(y^{\prime}))^{\ast}   \nonumber   \\
& + & \sum_{j=1}^{\infty} \frac{e^{-t \mu_{q-1, j}}}{\sqrt{4 \pi t}} \left( e^{- \frac{(u - v)^{2}}{4t}} + e^{- \frac{(u + v)^{2}}{4t}} \right)
(du \wedge \psi_{q-1, j}(y)) \otimes (dv \wedge \psi_{q-1, j}(y^{\prime}))^{\ast},   \nonumber
\end{eqnarray}

\begin{eqnarray}   \label{E:3.7}
{\mathcal E}^{\cyl, q}_{{\mathcal P}_{+, {\mathcal L}_{1}}} (t, (u, y), (v, y^{\prime}))
& = & \sum_{j=1}^{\infty} \frac{e^{-t \lambda_{q, j}}}{\sqrt{4 \pi t}} \left( e^{- \frac{(u - v)^{2}}{4t}} + e^{- \frac{(u + v)^{2}}{4t}} \right)
\phi_{q, j}(y) \otimes \phi^{\ast}_{q, j}(y^{\prime})     \\
& + & \sum_{j=1}^{\infty} \frac{e^{-t \mu_{q, j}}}{\sqrt{4 \pi t}} \left( e^{- \frac{(u - v)^{2}}{4t}} - e^{- \frac{(u + v)^{2}}{4t}} \right)
\psi_{q, j}(y) \otimes \psi^{\ast}_{q, j}(y^{\prime})   \nonumber   \\
& + & \sum_{j=1}^{\infty} \frac{e^{-t \lambda_{q-1, j}}}{\sqrt{4 \pi t}} \left( e^{- \frac{(u - v)^{2}}{4t}} + e^{- \frac{(u + v)^{2}}{4t}} \right)
(du \wedge \phi_{q-1, j}(y)) \otimes (dv \wedge \phi_{q-1, j}(y^{\prime}))^{\ast}   \nonumber   \\
& + & \sum_{j=1}^{\infty} \frac{e^{-t \mu_{q-1, j}}}{\sqrt{4 \pi t}} \left( e^{- \frac{(u - v)^{2}}{4t}} - e^{- \frac{(u + v)^{2}}{4t}} \right)
(du \wedge \psi_{q-1, j}(y)) \otimes (dv \wedge \psi_{q-1, j}(y^{\prime}))^{\ast}.   \nonumber
\end{eqnarray}

Let $\rho(a, b)$ be a smooth increasing function of real variable such that
\[ \rho(a, b) (u) = \left\{ \begin{array}{ll} 0 & \mbox{for $u \leq a$} \\
1 & \mbox{for $u \geq b$} \hspace{0.1 cm}.
\end{array} \right. \]
We put

\begin{eqnarray*}
\phi_{1} := 1 - \rho(\frac{5 \epsilon}{7}, \frac{6 \epsilon}{7}), \quad  \psi_{1} := 1 - \rho(\frac{3 \epsilon}{7}, \frac{4 \epsilon}{7}), \quad
 \phi_{2} := \rho(\frac{\epsilon}{7}, \frac{2 \epsilon}{7}), \quad  \psi_{2} := \rho(\frac{3 \epsilon}{7}, \frac{4 \epsilon}{7}),
\end{eqnarray*}

\noindent
and
\begin{eqnarray} \label{E:3.8}
{\mathcal Q}^{q}_{{\mathcal P}_{-, {\mathcal L}_{0}}}(t, (u, y), (v, y^{\prime}))  & = &
\phi_{1}(u) \E^{\cyl, q}_{{\mathcal P}_{-, {\mathcal L}_{0}}}(t, (u, y), (v, y^{\prime})) \psi_{1}(v) +
\phi_{2}(u) {\widetilde \E}^{q}(t, (u, y), (v, y^{\prime})) \psi_{2}(v),  \nonumber  \\
{\mathcal Q}^{q}_{{\mathcal P}_{+, {\mathcal L}_{1}}}(t, (u, y), (v, y^{\prime})) & = &
\phi_{1}(u) \E^{\cyl, q}_{{\mathcal P}_{+, {\mathcal L}_{1}}}(t, (u, y), (v, y^{\prime})) \psi_{1}(v) +
\phi_{2}(u) {\widetilde \E}^{q}(t, (u, y), (v, y^{\prime})) \psi_{2}(v).
\end{eqnarray}

\noindent
Then, ${\mathcal Q}^{q}_{{\mathcal P}_{-, {\mathcal L}_{0}}}$ and ${\mathcal Q}^{q}_{{\mathcal P}_{+, {\mathcal L}_{1}}}$ are
 parametrices for the kernels of $e^{-t \Delta^{q}_{{\mathcal P}_{-, {\mathcal L}_{0}}}}$ and
 $e^{-t \Delta^{q}_{{\mathcal P}_{+, {\mathcal L}_{1}}}}$, respectively.
The standard computation using (\ref{E:3.4}), (\ref{E:3.6}) and (\ref{E:3.7}) (see [2], [4] for details) shows that for $0 < t \leq 1$ and
$\alpha = {\mathcal P}_{-, {\mathcal L}_{0}}$ or ${\mathcal P}_{+, {\mathcal L}_{1}}$,
there exist some positive constants $c_{1}$ and $c_{2}$ such that

\begin{equation} \label{E:3.9}
| {\mathcal E}^{q}_{\alpha} (t, (u, y), (v, y^{\prime})) -
{\mathcal Q}^{q}_{\alpha} (t, (u, y), (v, y^{\prime})) | \leq c_{1} e^{- \frac{c_{2}}{t}},
\end{equation}

\noindent
which shows that

\begin{eqnarray} \label{E:3.10}
\lim_{t \rightarrow 0} \left( {\mathcal E}^{q}_{\alpha} (t, (u, y), (v, y^{\prime})) - {\mathcal Q}^{q}_{\alpha} (t, (u, y), (v, y^{\prime}))
\right) & = & 0.
\end{eqnarray}

\noindent
Hence, in view of (\ref{E:3.3}) with $x \in {\mathcal F}_{0}(f)$, we have

\begin{eqnarray} \label{E:3.11}
\lim_{t \rightarrow 0} \int_{U_{x}} \Tr \left( {\mathcal T}_{q}(x) {\mathcal E}^{q}_{\alpha}(t, f(x), x) \right) d vol(x) & = &
\lim_{t \rightarrow 0} \int_{U_{x}} \Tr \left( {\mathcal T}_{q}(x) {\mathcal Q}^{q}_{\alpha}(t, f(x), x) \right) d vol(x)   \nonumber   \\
& = & \lim_{t \rightarrow 0} \int_{U_{x}} \Tr \left( {\mathcal T}_{q}(x) {\widetilde {\mathcal E}}^{q}(t, f(x), x) \right) d vol(x),
\end{eqnarray}

\noindent
which yields the following equalities.

\begin{eqnarray} \label{E:3.12}
& & \lim_{t \rightarrow 0} \sum_{x \in {\mathcal F}_{0}(f)} \sum_{q= \even}
\int_{U_{x}} \Tr {\mathcal T}_{q}(x) \left( {\mathcal E}^{q}_{{\mathcal P}_{-, {\mathcal L}_{0}}}(t, f(x), x) \right) d vol(x)   \\
& & - \lim_{t \rightarrow 0} \sum_{x \in {\mathcal F}_{0}(f)}
\sum_{q= \odd} \int_{U_{x}} \Tr {\mathcal T}_{q}(x) \left( {\mathcal E}^{q}_{{\mathcal P}_{+, {\mathcal L}_{1}}}(t, f(x), x) \right) d vol(x)  \nonumber  \\
& = &  \lim_{t \rightarrow 0} \sum_{x \in {\mathcal F}_{0}(f)} \sum_{q=0}^{m} (-1)^{q}
\int_{U_{x}} \Tr \left( {\mathcal T}_{q}(x) {\widetilde {\mathcal E}}^{q} (t, f(x), x) \right) d vol(x)
\hspace{0.1 cm} = \hspace{0.1 cm} \sum_{x \in {\mathcal F}_{0}(f)} \Sign \ddet \left( I - df(x) \right),    \nonumber
\end{eqnarray}

\noindent
where we refer to Theorem 1.10.4 in [6] or Theorem 10.12 in [10] for the proof of the last equality.

We next analyze the boundary contribution.
For $\alpha = {\mathcal P}_{-, {\mathcal L}_{0}}$ or ${\mathcal P}_{+, {\mathcal L}_{1}}$, by (\ref{E:3.10}) we have

\begin{eqnarray} \label{E:3.13}
& & \lim_{t \rightarrow 0} \int_{Y} \int_{0}^{\frac{\epsilon}{7}}
\Tr \left( {\mathcal T}_{q}(x) {\mathcal E}^{q}_{\alpha}(t, f(x), x) \right) du \hspace{0.1 cm} d vol(y)
\hspace{0.1 cm} = \hspace{0.1 cm}
\lim_{t \rightarrow 0} \int_{Y} \int_{0}^{\frac{\epsilon}{7}}
\Tr \left( {\mathcal T}_{q}(x) {\mathcal Q}^{q}_{\alpha}(t, f(x), x) \right) du \hspace{0.1 cm} d vol(y)   \nonumber   \\
& = & \lim_{t \rightarrow 0} \int_{Y} \int_{0}^{\frac{\epsilon}{7}}
\Tr \left( {\mathcal T}_{q}(x) {\mathcal E}^{\cyl, q}_{\alpha}(t, f(x), x) \right) du \hspace{0.1 cm}
d vol(y).
\end{eqnarray}

\noindent
We note that on $[0, \frac{\epsilon}{7}) \times Y$, $f$ is assumed to be $f(u, y) = ( c \hspace{0.1 cm} u, B(y))$,
where $B : (Y, g^{Y}) \rightarrow (Y, g^{Y})$ is an isometry.
Let us consider the case of $\alpha = {\mathcal P}_{-, {\mathcal L}_{0}}$.
We can treat the case of $\alpha = {\mathcal P}_{+, {\mathcal L}_{1}}$ in the same way.
Put $x = ( u, y)$ and ${\frak B}_{q}(y) := \Lambda^{q} \left( ( d^{Y}B (y) )^{T} \right)$.
Since ${\mathcal T}_{q} (u, y)  \phi_{q, j} (B(y)) = {\frak B}_{q}(y) \phi_{q, j} (B(y))$, we have

\begin{eqnarray} \label{E:3.14}
& & \lim_{t \rightarrow 0} \int_{Y} \int_{0}^{\frac{\epsilon}{7}}
\sum_{j=1}^{\infty} \frac{e^{-t \lambda_{q, j}}}{\sqrt{4 \pi t}} \left( e^{- \frac{(c - 1)^{2} u^{2}}{4t}} - e^{- \frac{(c + 1)^{2} u^{2}}{4t}} \right)
\langle {\frak B}_{q}(y) \phi_{q, j}(B(y)), \hspace{0.1 cm} \phi_{q, j}(y) \rangle \hspace{0.1 cm} du \hspace{0.1 cm} d vol(y)  \nonumber  \\
& = & \lim_{t \rightarrow 0} \frac{1}{\sqrt{\pi}} \int_{0}^{\frac{\epsilon}{14 \sqrt{t}}}
\left( e^{- (c - 1)^{2} x^{2}} - e^{- (c + 1)^{2} x^{2}} \right) dx \cdot \lim_{t \rightarrow 0} \int_{Y} \sum_{j=1}^{\infty} e^{-t \lambda_{q, j}}
\langle {\frak B}_{q}(y) \phi_{q, j}(B(y)), \phi_{q, j}(y) \rangle d vol(y)  \nonumber  \\
& = & \frac{1}{2}  \left( \frac{1}{| 1 - c |} - \frac{1}{1 + c} \right) \cdot \lim_{t \rightarrow 0}
\Tr \left( B^{\ast} e^{-t \Delta_{Y}^{q}}|_{\Omega^{q}_{-}(Y)} \right),
\end{eqnarray}

\vspace{0.2 cm}

\noindent
where $\langle \hspace{0.1 cm} , \hspace{0.1 cm} \rangle$ is the pointwise inner product of differential forms induced by the metric $g_{Y}$.
Similarly, since ${\mathcal T}_{q}(u, y) \left( du \wedge \phi_{q-1, j} (B(y)) \right) =
c \hspace{0.1 cm} du \wedge \left( {\frak B}(y) \phi_{q-1, j}(B(y)) \right)$, we have

\begin{eqnarray} \label{E:3.15}
& & \lim_{t \rightarrow 0} \hspace{0.1 cm}  \int_{Y} \int_{0}^{\frac{\epsilon}{7}}
\sum_{j=1}^{\infty} \frac{e^{-t \lambda_{q, j}}}{\sqrt{4 \pi t}} \left( e^{- \frac{(c - 1)^{2} u^{2}}{4t}} - e^{- \frac{(c + 1)^{2} u^{2}}{4t}} \right) \times \nonumber  \\
& & \hspace{0.3 cm} \langle c du \wedge {\frak B}_{q-1}(y) \phi_{q-1, j} (B(y)), \hspace{0.1 cm}
 du \wedge \phi_{q-1, j}(y) \rangle \hspace{0.1 cm} du \hspace{0.1 cm} d vol(y)  \nonumber  \\
& = &  \frac{1}{2}  \left( \frac{c}{| 1 - c |} - \frac{c}{1 + c} \right) \cdot \lim_{t \rightarrow 0}
\Tr \left( B^{\ast} e^{-t \Delta_{Y}^{q-1}}|_{\Omega^{q-1}_{-}(Y)} \right).
\end{eqnarray}

\vspace{0.2 cm}

\noindent
Same computation using (\ref{E:3.6}) shows that

\begin{eqnarray} \label{E:3.16}
& & \lim_{t \rightarrow 0} \int_{Y} \int_{0}^{\frac{\epsilon}{7}} \Tr \left( {\mathcal T}_{q}(u, y) {\mathcal E}^{\cyl, q}_{{\mathcal P}_{-, {\mathcal L}_{0}}}
(t, f(u, y), (u, y)) \right) du \hspace{0.1 cm} d vol(y)     \\
& = & \frac{1}{2}  \left( \frac{1}{| 1 - c |} - \frac{1}{1 + c} \right) \cdot \lim_{t \rightarrow 0}
\Tr \left( B^{\ast} e^{-t \Delta_{Y}^{q}}|_{\Omega^{q}_{-}(Y)} \right)
\hspace{0.1 cm} + \hspace{0.1 cm}
\frac{1}{2}  \left( \frac{1}{| 1 - c |} + \frac{1}{1 + c} \right) \cdot \lim_{t \rightarrow 0}
\Tr \left( B^{\ast} e^{-t \Delta_{Y}^{q}}|_{\Omega^{q}_{+}(Y)} \right)  \nonumber  \\
& + & \frac{1}{2}  \left( \frac{c}{| 1 - c |} - \frac{c}{1 + c} \right) \cdot \lim_{t \rightarrow 0}
\Tr \left( B^{\ast} e^{-t \Delta_{Y}^{q-1}}|_{\Omega^{q-1}_{-}(Y)} \right)  \nonumber  \\
& + & \frac{1}{2}  \left( \frac{c}{| 1 - c |} + \frac{c}{1 + c} \right) \cdot \lim_{t \rightarrow 0}
\Tr \left( B^{\ast} e^{-t \Delta_{Y}^{q-1}}|_{\Omega^{q-1}_{+}(Y)} \right)  \nonumber  \\
& = & \frac{1}{2 | 1 - c |} \cdot \lim_{t \rightarrow 0} \Tr \left( B^{\ast} e^{-t \Delta_{Y}^{q}} \right)
\hspace{0.1 cm} + \hspace{0.1 cm} \frac{c}{2 | 1 - c |} \cdot \lim_{t \rightarrow 0} \Tr \left( B^{\ast} e^{-t \Delta_{Y}^{q-1}} \right) \nonumber  \\
&  & \hspace{0.1 cm} + \hspace{0.1 cm}
\frac{1}{2 (1 + c)} \cdot \lim_{t \rightarrow 0} \left( \Tr \left( B^{\ast} e^{-t \Delta_{Y}^{q}}|_{\Omega^{q}_{+}(Y)} \right) \hspace{0.1 cm} - \hspace{0.1 cm}
\Tr \left( B^{\ast} e^{-t \Delta_{Y}^{q}}|_{\Omega^{q}_{-}(Y)} \right) \right)  \nonumber  \\
&  & \hspace{0.1 cm} + \hspace{0.1 cm}
\frac{c}{2 (1 + c)} \cdot \lim_{t \rightarrow 0} \left( \Tr \left( B^{\ast} e^{-t \Delta_{Y}^{q-1}}|_{\Omega^{q-1}_{+}(Y)} \right) \hspace{0.1 cm} - \hspace{0.1 cm}
\Tr \left( B^{\ast} e^{-t \Delta_{Y}^{q-1}}|_{\Omega^{q-1}_{-}(Y)} \right) \right) .  \nonumber
\end{eqnarray}

\noindent
Similarly, using (\ref{E:3.7}), we have

\begin{eqnarray} \label{E:3.17}
& & \lim_{t \rightarrow 0} \int_{Y} \int_{0}^{\frac{\epsilon}{7}} \Tr \left( {\mathcal T}_{q}(u, y) {\mathcal E}^{\cyl, q}_{{\mathcal P}_{+, {\mathcal L}_{1}}}
(t, f(u, y), (u, y)) \right) du \hspace{0.1 cm} d vol(y)     \\
& = & \frac{1}{2 | 1 - c |} \cdot \lim_{t \rightarrow 0} \Tr \left( B^{\ast} e^{-t \Delta_{Y}^{q}} \right)
\hspace{0.1 cm} + \hspace{0.1 cm} \frac{c}{2 | 1 - c |} \cdot \lim_{t \rightarrow 0} \Tr \left( B^{\ast} e^{-t \Delta_{Y}^{q-1}} \right) \nonumber  \\
&  & \hspace{0.1 cm} - \hspace{0.1 cm}
\frac{1}{2 (1 + c)} \cdot \lim_{t \rightarrow 0} \left( \Tr \left( B^{\ast} e^{-t \Delta_{Y}^{q}}|_{\Omega^{q}_{+}(Y)} \right) \hspace{0.1 cm} - \hspace{0.1 cm}
\Tr \left( B^{\ast} e^{-t \Delta_{Y}^{q}}|_{\Omega^{q}_{-}(Y)} \right) \right)  \nonumber  \\
&  & \hspace{0.1 cm} - \hspace{0.1 cm}
\frac{c}{2 (1 + c)} \cdot \lim_{t \rightarrow 0} \left( \Tr \left( B^{\ast} e^{-t \Delta_{Y}^{q-1}}|_{\Omega^{q-1}_{+}(Y)} \right) \hspace{0.1 cm} - \hspace{0.1 cm}
\Tr \left( B^{\ast} e^{-t \Delta_{Y}^{q-1}}|_{\Omega^{q-1}_{-}(Y)} \right) \right) .  \nonumber
\end{eqnarray}

\vspace{0.2 cm}

\noindent
Finally, combining (\ref{E:3.16}) and (\ref{E:3.17}), we have

\begin{eqnarray} \label{E:3.18}
& & \lim_{t \rightarrow 0} \sum_{q = \even}
\int_{Y} \int_{0}^{\frac{\epsilon}{7}} \Tr \left( {\mathcal T}_{q}(u, y) {\mathcal E}^{\cyl, q}_{{\mathcal P}_{-, {\mathcal L}_{0}}}
(t, f(u, y), (u, y)) \right) du \hspace{0.1 cm} d vol(y) \nonumber  \\
& & \hspace{0.5 cm} - \hspace{0.1 cm} \lim_{t \rightarrow 0} \sum_{q = \odd}
\int_{Y} \int_{0}^{\frac{\epsilon}{7}} \Tr \left( {\mathcal T}_{q}(u, y) {\mathcal E}^{\cyl, q}_{{\mathcal P}_{+, {\mathcal L}_{1}}}
(t, f(u, y), (u, y)) \right) du \hspace{0.1 cm} d vol(y)  \nonumber  \\
& = & \frac{1-c}{2 | 1 - c | } \cdot \lim_{t \rightarrow 0} \sum_{q=0}^{m-1} (-1)^{q} \Tr \left( B^{\ast} e^{-t \Delta_{Y}^{q}} \right)  \nonumber  \\
& & \hspace{0.5 cm} + \hspace{0.1 cm} \lim_{t \rightarrow 0} \frac{1}{2} \sum_{q=0}^{m-1}
\left( \Tr \left( B^{\ast} e^{-t \Delta_{Y}^{q}}|_{\Omega^{q}_{+}(Y)} \right) \hspace{0.1 cm} - \hspace{0.1 cm}
\Tr \left( B^{\ast} e^{-t \Delta_{Y}^{q}}|_{\Omega^{q}_{-}(Y)} \right) \right) .
\end{eqnarray}

\noindent
Using (\ref{E:3.5}) and the following commutative diagram

\begin{eqnarray*}
\begin{CD}
  \Imm (d^{Y})^{\ast} \cap \Omega^{q}(Y)      & @> d^{Y} >>  & \Imm d^{Y} \cap \Omega^{q+1}(Y) \\
                     @V{B^{\ast} e^{-t \Delta_{Y}}}VV  \circlearrowright      &      & @VV {B^{\ast} e^{-t \Delta_{Y}}} V \\
  \Imm (d^{Y})^{\ast} \cap \Omega^{q}(Y)     & @> d^{Y} >>  &  \Imm d^{Y} \cap \Omega^{q+1}(Y)
\end{CD}
\end{eqnarray*}

\noindent
with the fact that $\Sign \ddet ( I - df(y) ) = \Sign ( 1 - c ) \cdot \Sign \ddet ( I - df_Y(y) )$,
we can rewrite (\ref{E:3.18}) by

\begin{eqnarray}  \label{E:3.19}
(\ref{E:3.18}) & = & \frac{1}{2} \sum_{y \in {\mathcal F}_{Y}(f)} \Sign \ddet ( I - df(y) )
+ \frac{1}{2} \left\{ \Tr \left( B^{\ast} : \left( \star_{Y} {\mathcal K} \right) \rightarrow \left( \star_{Y} {\mathcal K} \right) \right) -
\Tr \left( B^{\ast} : {\mathcal K} \rightarrow {\mathcal K} \right) \right\}.  \nonumber
\end{eqnarray}

\noindent
Furthermore, $\frac{1}{2} \left\{ \Tr \left( B^{\ast} : \left( \star_{Y} {\mathcal K} \right) \rightarrow \left( \star_{Y} {\mathcal K} \right) \right) -
\Tr \left( B^{\ast} : {\mathcal K} \rightarrow {\mathcal K} \right) \right\}$ ie equal to $0$
if $B : (Y, g^{Y}) \rightarrow (Y, g^{Y})$ is orientation preserving and is equal to
$\hspace{0.1 cm} - \Tr \left( B^{\ast} : {\mathcal K} \rightarrow {\mathcal K} \right)$ if $B$ is orientation reversing.
We can compute $L_{\widetilde{\mathcal P}_{1}}(f)$ in the same way.
Summarizing the above arguments with Lemma \ref{Lemma:2.2}, we have the following result, which is the main result of this paper.

\begin{theorem}  \label{Theorem:3.3}
Let $(M, Y, g^{M})$ be an $m$-dimensional compact oriented Riemannian manifold with boundary $Y$ and
$g^{M}$ be assumed to be a product metric near $Y$.
Suppose that $f : M \rightarrow M$ is a smooth map having only simple fixed points and satisfying the condition A.
Then the following equalities hold.

\begin{eqnarray*}
& (1) & \sum_{q= \even} \Tr \left( f^{\ast} : H^{q}(M, Y) \rightarrow H^{q}(M, Y) \right) -
\sum_{q= \odd} \Tr \left( f^{\ast} : H^{q}(M) \rightarrow H^{q}(M) \right)   \\
&  & \hspace{0.5 cm} = \hspace{0.1 cm}
\sum_{x \in {\mathcal F}_{0}(f)} \Sign \ddet ( I - df (x) ) + \frac{1}{2} \sum_{y \in {\mathcal F}_{Y}(f)} \Sign \ddet ( I - df(y) )
\hspace{0.1 cm} - \hspace{0.1 cm} K_{0} \\
& (2) & \sum_{q= \even} \Tr \left( f^{\ast} : H^{q}(M) \rightarrow H^{q}(M) \right) -
\sum_{q= \odd} \Tr \left( f^{\ast} : H^{q}(M, Y) \rightarrow H^{q}(M, Y) \right)   \\
&  &  \hspace{0.5 cm} = \hspace{0.1 cm}
\sum_{x \in {\mathcal F}_{0}(f)} \Sign \ddet ( I - df (x) ) + \frac{1}{2} \sum_{y \in {\mathcal F}_{Y}(f)} \Sign \ddet ( I - df(y) )
\hspace{0.1 cm} + \hspace{0.1 cm} K_{0},
\end{eqnarray*}

\noindent
where $K_{0} = 0$ if $B$ is orientation preserving and
$K_{0} = \Tr \left( B^{\ast} : \Imm \iota^{\ast} \rightarrow \Imm \iota^{\ast} \right)$ with $\iota^{\ast} : H^{\bullet}(M) \rightarrow H^{\bullet}(Y)$
if $B$ is orientation reversing.
\end{theorem}

\vspace{0.2 cm}

\noindent
Combining this result with (\ref{E:1.2}), we have the following result.

\begin{corollary}  \label{Corollary:3.5}
We assume the same assumptions as in Theorem \ref{Theorem:3.3}. Then :

\begin{eqnarray*}
& (1) & \sum_{q= \even} \Tr \left( f^{\ast} : H^{q}(M) \rightarrow H^{q}(M) \right) -
\sum_{q= \even} \Tr \left( f^{\ast} : H^{q}(M, Y) \rightarrow H^{q}(M, Y) \right)   \\
&  & \hspace{0.5 cm} = \hspace{0.1 cm}
\frac{1}{2} \sum_{y \in {\mathcal F}^{+}_{Y}(f)} \Sign \ddet ( I - df (y) ) - \frac{1}{2} \sum_{y \in {\mathcal F}^{-}_{Y}(f)} \Sign \ddet ( I - df(y) )
\hspace{0.1 cm} + \hspace{0.1 cm} K_{0}, \\
& (2) & \sum_{q= \odd} \Tr \left( f^{\ast} : H^{q}(M) \rightarrow H^{q}(M) \right) -
\sum_{q= \odd} \Tr \left( f^{\ast} : H^{q}(M, Y) \rightarrow H^{q}(M, Y) \right)   \\
&  & \hspace{0.5 cm} = \hspace{0.1 cm}
- \frac{1}{2} \sum_{y \in {\mathcal F}^{+}_{Y}(f)} \Sign \ddet ( I - df (y) ) + \frac{1}{2} \sum_{y \in {\mathcal F}^{-}_{Y}(f)} \Sign \ddet ( I - df(y) )
\hspace{0.1 cm} + \hspace{0.1 cm} K_{0},
\end{eqnarray*}

\noindent
where either ${\mathcal F}^{+}_{Y}(f) = \emptyset$ or ${\mathcal F}^{-}_{Y}(f) = \emptyset$, depending on $c$ in the Condition A.
\end{corollary}


\end{document}